\documentclass[11pt, letterpaper]{article}
\usepackage{amsmath, amsthm, amssymb, amsfonts} 
\usepackage{mathrsfs} 
\usepackage{bm} 
\usepackage{graphicx} 
\usepackage{enumerate} 
\usepackage{geometry} 
\usepackage{setspace} 
\usepackage{lmodern} 
\usepackage{hyperref} 
\usepackage{color} 
\usepackage{xcolor} 
\usepackage{url} 
\usepackage{mathtools}
\usepackage{enumitem}
\usepackage{changepage}
\usepackage{microtype}
\usepackage{authblk}
\usepackage{amsfonts}
\usepackage{mathrsfs,amscd,amssymb,amsthm,amsmath,bm,graphicx,psfrag,subfigure,url,mathtools}
\usepackage{pict2e}
\usepackage{psfrag,amsmath}
\usepackage{tikz}
\usepackage{indentfirst}
\usepackage{hyperref}
\usepackage{bookmark}
\usepackage{enumerate}
\usepackage{latexsym,euscript,epic,eepic,color}
\usepackage{multirow}
\usepackage{multicol}
\usepackage{longtable}
\usepackage{adjustbox}

\usepackage{setspace}
\usepackage{epstopdf}
\allowdisplaybreaks
\usepackage{authblk}
\usepackage{pifont}

\theoremstyle{plain}
\newtheorem{theorem}{Theorem}[section]          
\newtheorem{lemma}{Lemma}[section]              
\newtheorem{claim}{Claim} 

\theoremstyle{definition}

\newcommand{\ar}{\operatorname{ar}}
\newcommand{\ex}{\operatorname{ex}}

\newcommand{\calR}{\mathcal R}

\hypersetup{
    colorlinks=true,
    linkcolor=blue,
    citecolor=red,
    urlcolor=magenta,
}

\newcommand{\keywords}[1]{%
  \par\vspace{6pt}\noindent\textbf{Keywords: }#1\par
}

\newcommand{\MSC}[2][2020]{%
  \par\vspace{3pt}\noindent\textbf{MSC(#1): }#2\par
}

\title{Anti-Ramsey Number of Intersecting Odd Cycles}
\vspace{6mm}

\author{Haojie Zheng\thanks{Corresponding author. \\
\hspace*{2em}E-mail address: zhj9536@126.com (H. Zheng)}}

\affil{School of Mathematics and Statistics, and Hubei Key Lab--Math. Sci.,\linebreak Central China Normal University, Wuhan 430079, China}

\date{\today}

\allowdisplaybreaks

\begin{document}
\baselineskip=0.23in

\maketitle

\begin{abstract}
For a graph $H$, the anti-Ramsey number $\operatorname{ar}(n,H)$ is the maximum number of colors in an edge-coloring of $K_n$
containing no rainbow copy of $H$, where a copy is rainbow if its edges have pairwise distinct colors.
Let $s,t$ be nonnegative integers with $s+t\ge2$, and let $H_{s,t}$ be a graph consisting of $s$ triangles and $t$ odd
cycles of fixed lengths at least $5$, all sharing exactly one common vertex and otherwise pairwise vertex-disjoint.
Liu et al. (2024) determined $\operatorname{ar}(n,H_{s,0})$ for $s\ge3$ and $n\ge50s^2$.
In this paper, we determine the exact value of $\operatorname{ar}(n,H_{s,t})$ for every fixed $H_{s,t}$
with $t\ge1$ and all sufficiently large $n$.
\end{abstract}

\MSC{05C15; 05C35}

\section{Introduction}

All graphs are finite and simple. For a graph $G$, write $e(G)=|E(G)|$, and let $\delta(G)$, $\Delta(G)$ and
$\chi(G)$ denote its minimum degree, maximum degree and chromatic number, respectively.
A matching is a set of pairwise vertex-disjoint edges, and the matching number $\nu(G)$ is the maximum size of a matching in $G$.
For $A\subseteq V(G)$ and $x\in V(G)$, let $G[A]$ be the induced subgraph, and write $N_A(x)=N_G(x)\cap A$
and $d_A(x)=|N_A(x)|$, where $N_G(x)$ is the neighborhood of $x$. We abbreviate $E_G(A)=E(G[A])$ and
$e_G(A)=e(G[A])$.

For disjoint $A,B\subseteq V(G)$, let $E_G(A,B)$ be the set of edges between $A$ and $B$, and write
$e_G(A,B)=|E_G(A,B)|$ and $G[A,B]=(A\cup B,E_G(A,B))$.
A \textit{maximum cut} is a bipartition $V(G)=V_0\cup V_1$ maximizing $e_G(V_0,V_1)$.
For an edge set $M$, let $V(M)$ denote its endpoints, with $V(e)=\{u,v\}$ for an edge $e=uv$.
For a positive integer $k$, write $[k]=\{1,\ldots,k\}$.

For $A\subseteq V(G)$ and $x\in V(G)$, define $E_A(x)=\{uv\in E(G[A]):\{u,v\}\cap N_G(x)\ne\emptyset\}$,
and let $\nu(E_A(x))$ denote the maximum size of a matching contained in $E_A(x)$.

A graph is called \textit{$H$-free} if it contains no subgraph isomorphic to $H$.
The Tur\'an number $\operatorname{ex}(n,H)$ is the maximum number of edges in an $H$-free graph on $n$ vertices.
An $n$-vertex $H$-free graph with $\operatorname{ex}(n,H)$ edges is called an extremal graph for $H$, 
and the family of all such graphs is denoted by $\operatorname{EX}(n,H)$.
Let $T_2(n)$ be the complete bipartite graph on $n$ vertices whose part sizes differ by at most one, and put $t_2(n)=e(T_2(n))$.

An edge-colored graph is called rainbow if its edges have pairwise distinct colors.
For a fixed graph $H$, the anti-Ramsey number $\operatorname{ar}(n,H)$ is the maximum number of colors in an edge-coloring of $K_n$ containing no rainbow copy of $H$.
Thus,
\[
\operatorname{ar}(n,H)=
\max\bigl\{|c(E(K_n))|:c\text{ is an edge-coloring of }K_n \text{ containing no rainbow copy of }H \bigr\}.
\]
Anti-Ramsey problems were introduced by Erd\H{o}s, Simonovits, and S\'os~\cite{ESS1975}.

A useful connection between anti-Ramsey and Tur\'an problems is provided by representing graphs.
Given an edge-coloring $c$ of $K_n$, a representing graph is a spanning subgraph obtained by choosing exactly one
edge of each color. We denote the family of all representing graphs by $\mathcal R(c,K_n)$.
Every $G\in\mathcal R(c,K_n)$ is rainbow and satisfies $e(G)=|c(E(K_n))|$.
Consequently, if $c$ contains no rainbow copy of $H$, then every representing graph is $H$-free.
In particular, $\operatorname{ar}(n,H)\le \operatorname{ex}(n,H)$.

We consider graphs formed by odd cycles sharing a common vertex.
For integers $s\ge0$ and $t\ge1$, let $H_{s,t}$ be a graph consisting of $s$ triangles and $t$ odd cycles of lengths
at least $5$, all sharing exactly one common vertex and otherwise pairwise vertex-disjoint.
The common vertex is called the \textit{center} of $H_{s,t}$.

The Tur\'an problem for graphs of this form has been studied in several settings.
Let $F_r$ be the friendship graph consisting of $r$ triangles sharing a common vertex and otherwise vertex-disjoint.
Erd\H{o}s et al.~\cite{EFGG1995} determined $\operatorname{ex}(n,F_r)$ and characterized the extremal graphs
for sufficiently large $n$. Chen et al.~\cite{CGPW2003} obtained analogous results for intersecting cliques.
Hou et al.~\cite{HQL2016} subsequently determined the Tur\'an number and the extremal graphs for $r$ cycles of the same odd
length at least $5$ sharing a common vertex. 
They later determined the Tur\'an number and characterized all extremal graphs for $H_{s,t}$~\cite{HouQiuLiu}.

Related anti-Ramsey results have also been obtained for graphs whose constituent subgraphs share a common vertex.
Liu et al.~\cite{LLL2024} determined the anti-Ramsey numbers of friendship graphs.
Lu et al.~\cite{LLM2025} subsequently determined the corresponding numbers for graphs formed by cliques of the same order sharing
a common vertex. For further anti-Ramsey results, we refer to \cite{Alon1983,JW2003,JL2009,MBNL2005,WZLX2023}
for cycles and collections of vertex-disjoint triangles, and to~\cite{AJK2004} for even cycles in complete bipartite graphs. 
In my previous work, I determined the exact anti-Ramsey number of $tC_{2k+1}$, the disjoint union of $t$ cycles of length $2k+1$.
These results motivate the corresponding problem for intersecting odd cycles, allowing both triangles and longer odd cycles.
My main result is as follows.



\begin{theorem}\label{thm:main}
Let $s\geq 0$ and $t\geq 1$ be fixed integers with $s+t\geq2$. There exists $n_0(H_{s,t})\in \mathbb{N}$ such that for all $n\geq n_0(H_{s,t})$,
\[
  \ar(n, H_{s,t})= t_2(n) + (s+t-2)^2 + 1.
\]

\end{theorem}

The remainder of the paper is organized as follows. Section~2 records
the extremal and structural results used in the proof. Section~3 establishes a structural lemma for
$H_{s,t}$-free graphs with large minimum degree and proves the anti-Ramsey bound when every representing
graph has large minimum degree. Section~4 proves Theorem~\ref{thm:main}.

\section{Preliminaries}

Let $T_2(n)$ be the balanced complete bipartite graph on $n$ vertices, and let $t_2(n)=e(T_2(n))$. We shall repeatedly use
$t_2(n)-t_2(n-1)=\left\lfloor\frac n2\right\rfloor$.
Indeed, $T_2(n)$ is obtained from $T_2(n-1)$ by adding one vertex to a smallest part, and the new vertex has degree $\lfloor n/2\rfloor$.

For nonnegative integers $\nu$ and $\Delta$, define $f(\nu,\Delta)=\max\{e(J):\nu(J)\le\nu,\ \Delta(J)\le\Delta\}$.

\begin{lemma}[\cite{ChvatalHanson}]\label{thm:CH}
For every $\nu\geq 1$ and $\Delta\ge1$,
\[
  f(\nu,\Delta)=\nu\Delta+
  \left\lfloor\frac{\Delta}{2}\right\rfloor
  \left\lfloor
  \frac{\nu}{\lceil\Delta/2\rceil}
  \right\rfloor
  \le \nu(\Delta+1).
\]
\end{lemma}

\begin{lemma}[\cite{HouQiuLiu}]
\label{thm:HQL}
Let $s\ge0$ and $t\ge1$ be fixed integers. Then for sufficiently large $n$,
\[
  \ex(n, H_{s,t})=t_2(n)+(s+t-1)^2.
\]
Moreover, every extremal graph is obtained from $T_2(n)$ by embedding $K_{s+t-1,s+t-1}$
in one part, with the additional possibility of embedding $3K_3$ when $(s,t)=(3,1)$. 
\end{lemma}

\begin{lemma}[\cite{erdos1968}, \cite{simonovits1968}]\label{es}
Let $H$ be a graph with $\chi(H) = r \geq 3$ and $H \neq K_r$.
Then, for every $\varepsilon > 0$, there exist $\eta > 0$ and $n_0 = n_0(H, \varepsilon) \in \mathbb{N}$ such that the following holds.
If $G$ is an $H$-free graph on $n \ge n_0$ vertices with $e(G) \ge \operatorname{ex}(n, H) - \eta n^2$,
then there exists a partition of $V(G) = V_1 \cup \cdots \cup V_{r-1}$ such that
$\sum\limits_{i=1}^{r-1} e(V_i) <\varepsilon n^2$.
\end{lemma}

Set $\gamma=[400(c(H_{s,t})+1)(s+t)]^{-2}$, where $c(H_{s,t})$ denotes the circumference of $H_{s,t}$.

\begin{lemma}[\cite{HouQiuLiu}]\label{lem:211}
Let $s\ge 0$, $t\ge 1$ and $k=s+t$. Let $G$ be a graph on $n$
vertices satisfying the following:
\begin{enumerate}[label=\textup{(\roman*)}]
\item $V_0\cup V_1$ is a partition of $V(G)$ with
  $\max\{|V_0|,|V_1|\}\le\bigl(\tfrac12+\sqrt{\gamma}\bigr)n$,
  and, for each $i=0,1$, $V_i$ has a subset $B_i$ with
  $E(G[B_i])=\varnothing$ and
  $|B_0\cup B_1|<\sqrt{\gamma}\,n$, and
  \[
  d_{V_{1-i}}(u)>
  \begin{cases}
  2n/5, & \text{if }u\in V_i\setminus B_i,\\
  n/9,  & \text{if }u\in B_i;
  \end{cases}
  \]
\item $V_i$ has a subset $U_i$ with
  $|V_i\setminus U_i|<\sqrt{\gamma}\,n$ for $i=0,1$.

\item there exists a vertex $x\in B_i$ with $d_{V_i}(x)\ge k$, or a
vertex $x\in V_i\setminus B_i$ and a matching $M_{1-i}$ of
$G[N_{V_{1-i}}(x)\setminus B_{1-i}]$ with
$d_{V_i}(x)+|M_{1-i}|\ge s$ and
\begin{equation}\label{long}
d_{V_i}(x)+|M_{1-i}|+\nu\bigl(G[V_i\setminus N_{V_i}(x)]\bigr)+\nu\!\left(E_{V_{1-i}\setminus V(M_{1-i})}(x)\right)\ge k,
\end{equation}
\end{enumerate}
then for sufficiently large $n$, there is a copy of $H_{s,t}$, say
$H$, in $G$ centered at $x$, satisfying that
\begin{enumerate}[label=\textup{(\arabic*)}]
\item $H$ contains exactly $k$ edges in $E(V_0)\cup E(V_1)$,
\item $V(H)\cap U_{1-i}\ne\varnothing$,
\item if $d_{V_i}(x)\ge k$, then $V(H)\cap U_j\ne\varnothing$ for $j=0,1$.
\end{enumerate}
\end{lemma}

\section{The high minimum-degree case}

\begin{lemma}\label{lem:structure}
Let $H_{s,t}$ be fixed with $k=s+t\ge 3$. Let $G$ be an $H_{s,t}$-free graph on $n$ vertices satisfying $\delta(G)\ge \left\lfloor\frac{n}{2}\right\rfloor$. Then, for sufficiently large $n$, $G$ admits a bipartition $V(G)=V_0\cup V_1$
such that $\Delta(G[V_i])\le k-1$ and $\nu(G[V_i])\le k-1$ for every $i\in\{0,1\}$.
\end{lemma}

\begin{proof}
We prove the lemma through the following three claims.

\begin{claim}\label{bipa}
There exists a bipartition $V(G)=V_0\cup V_1$ such that $e(G[V_0])+e(G[V_1])<\frac{\gamma n^2}{4}$  and
$E_G(V_0,V_1)$ is a maximum cut of $G$, where $\gamma=[400(c(H_{s,t})+1)k]^{-2}$.
\end{claim}
\begin{proof}[\bf Proof of Claim~\ref{bipa}]
Since $\delta(G)\ge \left\lfloor\frac{n}{2}\right\rfloor$, we obtain
\begin{equation}\label{eq:edge-lower-bound}
e(G)=\frac12\sum_{v\in V(G)}d_G(v)
\ge \frac n2\left\lfloor\frac n2\right\rfloor
\ge \frac{n^2}{4}-\frac n4.
\end{equation}
By Lemma~\ref{thm:HQL}, for sufficiently large $n$,
\begin{equation}\label{eq:extremal-number}
\operatorname{ex}(n,H_{s,t})
=\left\lfloor\frac{n^2}{4}\right\rfloor+(k-1)^2
\le \frac{n^2}{4}+(k-1)^2.
\end{equation}
As $G$ is $H_{s,t}$-free, combining (\ref{eq:edge-lower-bound}) and (\ref{eq:extremal-number}) gives
$0\le \operatorname{ex}(n,H_{s,t})-e(G)\le \frac n4+(k-1)^2$. Hence, for every fixed $\eta>0$ and sufficiently
large $n$, we have 
\[
e(G)\geq \operatorname{ex}(n,H_{s,t})-\frac n4-(k-1)^2\geq\operatorname{ex}(n,H_{s,t})-\eta n^2.
\]
Consequently, for every fixed $\varepsilon>0$ and sufficiently large $n$, Lemma~\ref{es} gives a partition $V(G)=V_0\cup V_1$ such that $e(G[V_0])+e(G[V_1])<\varepsilon n^2$. 

Since $e(G)$ is fixed and $e(G[V_0])+e(G[V_1])<\varepsilon n^2$, we have $e_G(V_0,V_1)>e(G)-\varepsilon n^2$.
Passing to a maximum cut does not increase the total number of internal edges. Hence we may assume that
$E_G(V_0,V_1)$ is a maximum cut of $G$.
Taking $\varepsilon=\gamma/4$ completes the proof. 
\end{proof}

\begin{claim}\label{deta}
$\Delta(G[V_i])\le k-1$ for every $i\in\{0,1\}$.
\end{claim}
\begin{proof}[\bf Proof of Claim~\ref{deta}]
We prove the claim using Lemma~\ref{lem:211}. To this end, we first verify that $G$ satisfies Lemma~\ref{lem:211}~(i), (ii).  Condition~(ii) of Lemma~\ref{lem:211} is immediately satisfied by taking $U_j=V_j$ for $j\in\{0,1\}$.

By Claim~\ref{bipa}, $|V_0|+|V_1|=n$ and $e(G)-(e(G[V_0])+e(G[V_1]))\le |V_0||V_1|$. Using \eqref{eq:edge-lower-bound} and the bound
$e(G[V_0])+e(G[V_1])<\gamma n^2/4$, we have
\[
\begin{aligned}
\left(|V_i|-\frac n2\right)^2
=\frac{n^2}{4}-|V_i|\bigl(n-|V_i|\bigr)
=\frac{n^2}{4}-|V_i||V_{1-i}|
&=\frac{n^2}{4}-|V_0||V_1|\\
&\le \frac{n^2}{4}-e(G)+e(G[V_0])+e(G[V_1])\\
&\le \frac n4+\gamma n^2/4<\gamma n^2,
\end{aligned}
\]
for each $i\in\{0,1\}$. It follows that $\max\{|V_0|,|V_1|\}\le \left(\frac12+\sqrt{\gamma}\right)n$.

Set $\beta=(c(H_{s,t})+1)\sqrt{\gamma}=\frac{1}{400k}$ so that $\sqrt{\gamma}<\beta<1/10$.
For each $i\in\{0,1\}$, define \\$B_i=\{x\in V_i:d_{V_i}(x)>\beta n\}$.
Since $B_0$ and $B_1$ are disjoint, we have
\[
\begin{aligned}
\beta n|B_0\cup B_1|
=\beta n\bigl(|B_0|+|B_1|\bigr)
\le \sum_{i=0}^{1}\sum_{x\in B_i}d_{V_i}(x)
\le \sum_{i=0}^{1}\sum_{x\in V_i}d_{V_i}(x)
=2\bigl(e(G[V_0])+e(G[V_1])\bigr).
\end{aligned}
\]
Using $e(G[V_0])+e(G[V_1])<\gamma n^2/4$ and $\beta=(c(H_{s,t})+1)\sqrt{\gamma}$, we obtain
\[
|B_0\cup B_1|
\le \frac{2(e(G[V_0])+e(G[V_1]))}{\beta n}
<\frac{\gamma n}{2\beta}
=\frac{\sqrt{\gamma}\,n}{2(c(H_{s,t})+1)}
<\sqrt{\gamma}\,n.
\]

For each $i\in\{0,1\}$ and every $x\in V_i\setminus B_i$, the definition of $B_i$ gives $d_{V_i}(x)\le\beta n$.
Together with $\delta(G)\ge \left\lfloor\frac{n}{2}\right\rfloor$, this yields
\[
\begin{aligned}
d_{V_{1-i}}(x)
=d_G(x)-d_{V_i}(x)
\ge \left\lfloor\frac n2\right\rfloor-\beta n
&\ge \frac {n-1}{2}-\beta n\\
&=\frac{2n}{5}
  +\left(\frac{1}{10}-\beta\right)n-\frac12\\
&>\frac{2n}{5}.
\end{aligned}
\]

Note that $E_{G}(V_0, V_1)$ is a maximum cut of $G$. By the maximality of the cut, we have $d_{V_{1-i}}(x)\ge d_{V_i}(x)$
for every $x\in V_i$, $i\in\{0,1\}$.
Consequently, $d_{V_{1-i}}(x)\ge \frac12 d_G(x)\ge \frac12\left\lfloor\frac n2\right\rfloor>\frac n9.$

Next we show that $B_0=B_1=\emptyset$.
Suppose otherwise, and let $G'$ be obtained from $G$ by deleting all edges of $G[B_0]$ and $G[B_1]$. 
Then $G'$ is $H_{s,t}$-free, and the sets $B_0$ and $B_1$ are independent in $G'$. 
At this point, $G'$ satisfies  conditions~(i) and~(ii) of Lemma~\ref{lem:211}. 
Since $B_0\cup B_1\ne\emptyset$, choose $i\in\{0,1\}$ and $x\in B_i$.
By the definition of $B_i$ and the bound $|B_i|<\sqrt{\gamma}\,n$, we have
\[
d_{G'[V_i]}(x)
\ge d_{V_i}(x)-|B_i|
>(\beta-\sqrt{\gamma})n
\geq k
\]
for sufficiently large $n$. Thus condition~(iii) of Lemma~\ref{lem:211} holds. It follows that there is a copy of $H_{s,t}$ centered at $x$, contradicting that $G'$ is $H_{s,t}$-free. Hence $B_0=B_1=\emptyset$, which implies that $d_{V_{1-i}}(u)> \frac{2n}{5}$ for  every $i\in\{0,1\}$ and every $u\in V_i$.

Suppose that $d_{V_i}(x)\ge k$ for some $i\in\{0,1\}$ and $x\in V_i$.
For this application of Lemma~\ref{lem:211}, take $\widetilde B_i=\{x\}$, $\widetilde B_{1-i}=\emptyset$. 
These sets are independent and satisfy $|\widetilde B_0\cup \widetilde B_1|=1<\sqrt{\gamma}\,n$
for sufficiently large $n$. Together with $d_{V_{1-i}}(u)> \frac{2n}{5}$ for  every $i\in\{0,1\}$ and every $u\in V_i$, this verifies condition~(i) of  Lemma~\ref{lem:211} with $\widetilde B_0$ and $\widetilde B_1$. 

Note that condition~(ii) holds with $U_j=V_j$ for $j\in\{0,1\}$. Since $x\in \widetilde B_i$ and $d_{V_i}(x)\ge k$, condition~(iii) of  Lemma~\ref{lem:211} holds. Therefore  there is a copy of $H_{s,t}$ centered at $x$, contradicting the assumption that $G$ is $H_{s,t}$-free.
Hence $\Delta(G[V_i])\le k-1$ for every $i\in\{0,1\}$.

\end{proof}

\begin{claim}\label{matching}
$\nu(G[V_i])\le k-1$ for every $i\in\{0,1\}$.
\end{claim}

\begin{proof}[\bf Proof of Claim~\ref{matching}]

Suppose for a contradiction that the claim is false. Without loss of generality, assume that $G[V_1]$ contains a matching $M_1$ consisting of exactly $k$ edges.

Similarly, we shall apply Lemma~\ref{lem:211} to show that $G$ admits  a copy of $H_{s,t}$.  To this end, we first verify that $G$ satisfies Lemma~\ref{lem:211}~(i), (ii).

By Claim~\ref{bipa} and Claim~\ref{deta}, $G$ has a bipartition $V_0\cup V_1$ with $\max\{|V_0|,|V_1|\}\le\bigl(\tfrac12+\sqrt{\gamma}\bigr)n$, and
every vertex of $G$ has more than $2n/5$ neighbors in the opposite part. Set $B_0=B_1=\emptyset$ and $U_j=V_j$ for $j\in\{0,1\}$. Hence conditions (i) and (ii) of Lemma~\ref{lem:211} hold.

By Claim~\ref{deta} and $\delta(G)\ge \left\lfloor n/2\right\rfloor$, we have
$d_{V_{1-j}}(u)=d_G(u)-d_{V_j}(u)\ge \left\lfloor\frac n2\right\rfloor-k+1$
for $j\in\{0,1\}$ and every $u\in V_j$.
It follows that $|V_j|\geq \left\lfloor\frac n2\right\rfloor-k+1$ for $j\in\{0,1\}$. 
Since $|V_0|+|V_1|=n$, we obtain
\begin{equation}\label{V}
\left\lfloor\frac n2\right\rfloor-k+1
\le |V_j|
\le \left\lceil\frac n2\right\rceil+k-1.
\end{equation}
Consequently, every $u\in V_j$ has at most
\begin{equation}\label{non}
|V_{1-j}|-d_{V_{1-j}}(u)
\le
\left\lceil\frac n2\right\rceil
-\left\lfloor\frac n2\right\rfloor+2k-2
\le 2k-1
\end{equation}
non-neighbors in the opposite part. 

Note that $V(M_1)\subseteq V_1$. By (\ref{V}) and (\ref{non}), we have
\[
\begin{aligned}
\left|\bigcap_{v\in V(M_1)}N_{V_0}(v)\right|
&\ge |V_0|
-\sum_{v\in V(M_1)}
  \bigl(|V_0|-d_{V_0}(v)\bigr)
\ge \left\lfloor\frac n2\right\rfloor-k+1
      -2k(2k-1)>0
\end{aligned}
\]
for sufficiently large $n$. Choose a vertex $z_0\in \bigcap\limits_{v\in V(M_1)}N_{V_0}(v)$. Then $z_0\in V_0\setminus B_0$ and $M_1$ is a
matching in $G[N_{V_1}(z_0)\setminus B_1]$. Note that $|M_1|=k$. It is easy to see that 
$d_{V_0}(z_0)+|M_{1}|\ge k\geq s$. Moreover, \eqref{long} holds for $G$ with center $z_0$ and matching $M_1$.
Thus $G$ satisfies condition (iii) of Lemma~\ref{lem:211}.

Lemma~\ref{lem:211} therefore yields a copy of $H_{s,t}$ centered at $z_0$, contradicting the assumption that $G$ is $H_{s,t}$-free.
Hence $\nu(G[V_i])\le k-1$ for every $i\in\{0,1\}$.
\end{proof}
\end{proof}

\begin{lemma}\label{lem:high-degree}
Fix $H_{s,t}$ with $k=s+t\ge3$. Suppose that $c$ is an edge-coloring of $K_n$ containing no rainbow copy of
$H_{s,t}$ and $n\geq n_1(H_{s,t})$. If every representing graph $G\in\calR(c,K_n)$ satisfies $\delta(G)\ge\left\lfloor\frac n2\right\rfloor$,
then $|c(E(K_n))|\leq t_2(n)+(k-2)^2+1$.
\end{lemma}

\begin{proof}
Suppose to the contrary that $|c(E(K_n))|\ge t_2(n)+(k-2)^2+2$.
Choose $G\in\calR(c,K_n)$. Then $G$ is rainbow and $H_{s,t}$-free, and  $\delta(G)\ge\lfloor \frac n2\rfloor$.
By Lemma~\ref{lem:structure}, there is a partition
$V(G)=V_0\cup V_1$ such that $\Delta(G[V_i])\le k-1$ and $\nu(G[V_i])\le k-1$ for each $i\in\{0,1\}$.
Consequently, by Lemma~\ref{thm:CH}, we have $e(G[V_i])\le f(k-1,k-1)$.

As shown in the proof of Lemma~\ref{lem:structure}, we have
\begin{equation}\label{eq:high-part-sizes}
\left\lfloor\frac n2\right\rfloor-k+1
\le |V_i|
\le \left\lceil\frac n2\right\rceil+k-1
\qquad(i\in\{0,1\}).
\end{equation}
Since $|V_0||V_1|\le t_2(n)$ and $e(G)=|c(E(K_n))|\ge t_2(n)+(k-2)^2+2$, we obtain
\[
\begin{aligned}
|V_0||V_1|-e_G(V_0,V_1)
&=|V_0||V_1|-e(G)+e(G[V_0])+e(G[V_1])\\
&\le e(G[V_0])+e(G[V_1])-(k-2)^2-2\\
&\le 2f(k-1,k-1).
\end{aligned}
\]

For each $i\in\{0,1\}$, define $S_i=\{x\in V_i:V_{1-i}\subseteq N_G(x)\}$.
By definition of $S_i$, every $x\in V_i\setminus S_i$ satisfies $|V_{1-i}|-d_{V_{1-i}}(x)\ge1$.
Therefore,
\[
\begin{aligned}
|V_i\setminus S_i|
\le \sum_{x\in V_i}
\bigl(|V_{1-i}|-d_{V_{1-i}}(x)\bigr)
=|V_0||V_1|-e_G(V_0,V_1)
\le 2f(k-1,k-1).
\end{aligned}
\]
Together with \eqref{eq:high-part-sizes}, this gives $|S_i|=\frac n2+O_k(1)$.

\begin{claim}\label{cl:high-rainbow-matching}
For each $i\in\{0,1\}$, the coloring induced on
$K_n[S_i]$ contains no rainbow matching of size $k$.
\end{claim}

\begin{proof}[\bf Proof of Claim~\ref{cl:high-rainbow-matching}]
Suppose otherwise. By symmetry, assume that $K_n[S_1]$ contains a rainbow matching $M=\{e_1,\ldots,e_k\}$.
For each $j\in[k]$, let $g_j$ be the unique edge of $G$ with color $c(e_j)$, and let
$G^*=\bigl(G-\{g_1,\ldots,g_k\}\bigr)\cup\{e_1,\ldots,e_k\}$.
Since the colors of $e_1,\ldots,e_k$ are distinct, $G^*$ is a representing graph.

We shall apply Lemma~\ref{lem:211} to show that $G^*$ admits  a copy of $H_{s,t}$.  To this end, we first verify that $G^*$ satisfies Lemma~\ref{lem:211}~(i), (ii). Condition~(ii) of Lemma~\ref{lem:211} is immediately satisfied by taking $U_j=V_j$ for $j\in\{0,1\}$.

Take $B_0=B_1=\emptyset$ in Lemma~\ref{lem:211}. These sets are independent and satisfy $|B_0\cup B_1|=0<\sqrt{\gamma}\,n$.
By \eqref{eq:high-part-sizes}, for sufficiently large $n$, we have
\[
\max\{|V_0|,|V_1|\}
\le \left\lceil\frac n2\right\rceil+k-1
\le \left(\frac12+\sqrt{\gamma}\right)n, 
\]
where $\gamma=[400(c(H_{s,t})+1)k]^{-2}$.

Furthermore, using $|V_0||V_1|-e_G(V_0,V_1)\le  2f(k-1,k-1)$ and the fact that at most $k$ edges of $G$ are deleted, for sufficiently large $n$, we obtain
\[
\begin{aligned}
|N_{G^*}(u)\cap V_{1-j}|
\ge |N_G(u)\cap V_{1-j}|-k
&\ge |V_{1-j}|-2f(k-1,k-1)-k\\
&\ge \left\lfloor\frac n2\right\rfloor-2f(k-1,k-1)-2k+1\\
&>\frac{2n}{5}
\end{aligned}
\]
for $j\in\{0,1\}$ and every $u\in V_j$. Thus condition~(i) of Lemma~\ref{lem:211} holds.

By the definition of $S_1$, every vertex of $V_0$ is adjacent to every endpoint of $e_1,\ldots,e_{k}$ in $G$. A vertex of $V_0$ can lose this property in $G^*$ only if it is incident with a deleted edge in $E_{G}(V_0,V_1)$. Hence at most $k$ vertices of $V_0$ lose this property. Since $|V_0|\geq\left\lfloor\frac n2\right\rfloor-k+1>k$ for sufficiently large $n$, we can choose a vertex $z\in V_0$ adjacent in $G^*$ to both endpoints of every $e_i$ for $i\in[k]$.

By the choice of $z$, we have $z\in V_0\setminus B_0$, and $M$ is a matching in $G^*[(N_{G^*}(z)\cap V_1)\setminus B_1]$.
Since $|M|=k$, we have  $d_{G^*[V_0]}(z)+|M|\ge k\ge s$. 
Moreover, \eqref{long} holds for $G^*$ with center $z$ and matching $M$. Hence condition~(iii) is satisfied.

Applying Lemma~\ref{lem:211} to $G^*$ with center $z$, we obtain a copy of $H_{s,t}$.
Since $G^*$ is a representing graph, this copy is rainbow, contradicting the assumption on $c$.
\end{proof}

\begin{claim}\label{mono}
There exists a constant $c_1=c_1(k)$ such that, for every $i\in\{0,1\}$, the colored complete graph
$K_n[S_i]$ contains a monochromatic matching $M_i$ satisfying $|M_i|\ge\frac{n/2-c_1}{4(k-1)}$.
\end{claim}

\begin{proof}[\bf Proof of Claim~\ref{mono}]

For each $i\in\{0,1\}$, choose a maximal rainbow matching $I_i$ in $K_n[S_i]$, and let $T_i=S_i\setminus V(I_i)$.
By Claim~\ref{cl:high-rainbow-matching}, $|I_i|\le k-1$, so there exists a constant $c_1=c_1(k)$ such that
$|T_i|=|S_i|-|V(I_i)|\geq \frac n2-c_1$.

By the maximality of $I_i$, every color appearing on $K_n[T_i]$ already appears on $I_i$.
Hence at most $k-1$ colors occur on $K_n[T_i]$. By the pigeonhole principle, there exists a color $\alpha_i$ that appears on at least
$\frac{1}{k-1}\binom{|T_i|}{2}$ edges of $K_n[T_i]$.

Let $F_i$ be the spanning subgraph of $K_n[T_i]$ formed by the edges of color $\alpha_i$, and let $M_i$ be a maximum matching in $F_i$.
Since $V(M_i)$ is a vertex cover of $F_i$,
\[
2|M_i|(|T_i|-1)
\ge e(F_i)
\ge \frac{1}{k-1}\binom{|T_i|}{2}.
\]
Therefore,
\[
|M_i|\ge \frac{|T_i|}{4(k-1)}\ge\frac{n/2-c_1}{4(k-1)}.
\]
Hence $M_i$ is a monochromatic matching of the required size.
\end{proof}

For each $i\in\{0,1\}$, choose a monochromatic matching $M_i$ in $K_n[S_i]$ as in Claim~\ref{mono}, and denote its color by $\alpha_i$.
Let $G_0$ be obtained from $G$ by deleting the edges of colors $\alpha_0$ and $\alpha_1$,
deleting only one edge if $\alpha_0=\alpha_1$. Note that the graph $G$ is a representing graph.
Then
\begin{equation}\label{G_0}
e(G_0)\ge e(G)-2\ge t_2(n)+(k-2)^2.
\end{equation}

Let $v$ be the center of $H_{s,t}$, and let $C$ be a shortest cycle of $H_{s,t}$ with length $l$.
Define  $\widehat H=H_{s,t}-(V(C)\setminus\{v\})$.
Then $\widehat H$ has $k-1$ cycles and contains at least one odd cycle of length at least $5$.

\begin{claim}\label{cl:high-deleted-branch}
The graph $G_0$ is $\widehat H$-free.
\end{claim}

\begin{proof}[\bf Proof of Claim~\ref{cl:high-deleted-branch}]
Suppose to the contrary that $G_0$ contains a copy $R$ of $\widehat H$ with center $u$. By symmetry, assume that $u\in V_0$. 
Note that $G_0$ is obtained from $G$ by deleting at most two edges. Denote by $L$ the set of all endpoints of the deleted edges. Thus, $|L|\leq4$.

By Claim~\ref{mono}, $M_1$ is a monochromatic matching of color $\alpha_1$ satisfying $|M_1|\ge\frac{n/2-c_1}{4(k-1)}$. Note that $|V(R)\cup L|$ is bounded independently of $n$. Hence, for sufficiently large $n$, we can choose an edge $xy\in M_1$ such that $\{x,y\}\cap\bigl(V(R)\cup L\bigr)=\varnothing$.

By the definition of $S_i$, every vertex of $S_i\setminus L$ is adjacent in $G_0$ to every
vertex of $V_{1-i}$. In particular, $ux,uy\in E(G_0)$. Moreover, no edge of $G_0$ has color $\alpha_1$.
Since $c(xy)=\alpha_1$, the graph $G_0+xy$ is rainbow.

We now construct a cycle $C'$ of length $l$ such that $V(C')\cap V(R)=\{u\}$.
If $l=3$, take $C'=uxyu$.
If $l\ge5$, put $r=(l-3)/2$.
Since $|S_i|=\frac n2+O_k(1)$ and $l$ is fixed, for sufficiently large $n$ we can choose distinct vertices
$a_1,\ldots,a_r\in S_0\setminus\bigl(V(R)\cup L\bigr)$ and $b_1,\ldots,b_r\in S_1\setminus\bigl(V(R)\cup L\cup\{x,y\}\bigr)$.
Then
\[
C'=uxya_1b_1a_2b_2\cdots a_rb_ru
\]
is a cycle of length $l$ in $G_0+xy$.
Indeed, every edge of this cycle other than $xy$ joins a vertex of $S_i\setminus L$ to a vertex of
$V_{1-i}$ for some $i\in\{0,1\}$, and hence belongs to $G_0$.

In both cases, $V(C')\cap V(R)=\{u\}$. Thus $R\cup C'$ is a copy of $H_{s,t}$ in the
rainbow graph $G_0+xy$, contradicting the assumption on $c$.
Therefore, $G_0$ is $\widehat H$-free.
\end{proof}

By (\ref{G_0}), Claim~\ref{cl:high-deleted-branch} and Lemma~\ref{thm:HQL},
\[
t_2(n)+(k-2)^2
\le e(G_0)
\le \operatorname{ex}(n,\widehat H)
=t_2(n)+(k-2)^2.
\]
Hence $e(G_0)=t_2(n)+(k-2)^2$.
Furthermore, $\alpha_0\ne\alpha_1$.

Thus $G_0\in EX(n,\widehat H)$. By Lemma~\ref{thm:HQL}, $G_0$ is obtained from $T_2(n)$ by embedding a graph $J$ into one of its parts.
Here $J\cong K_{k-2,k-2}$, with the additional possibility $J\cong3K_3$ when $\widehat H=H_{3,1}$. In either case, $\nu(J)=k-2$.

For convenience, let $W_0\cup W_1$ be the balanced partition of $V(T_2(n))$. Clearly, $W_0\cup W_1$ is also a balanced partition of $V(G_0)$. 
The claim below describes the relationship between the new partition $W_0\cup W_1$ of $V(G_0)$ and its original partition $V_0\cup V_1$.

\begin{claim}\label{partitions}
There exists a permutation $\sigma$ of $\{0,1\}$ such that $V_i=W_{\sigma(i)}$ for every $i\in\{0,1\}$.
\end{claim}

\begin{proof}[\bf Proof of Claim~\ref{partitions}]
By the bound on $e(G[V_i])$ established above,
we have
\begin{equation}\label{eq:high-partition-internal}
e(G_0[V_i])\le e(G[V_i])\le f(k-1,k-1)
\qquad\text{for every }i\in\{0,1\}.
\end{equation}

For each $i\in\{0,1\}$, choose $\sigma(i)\in\{0,1\}$ such that $|V_i\cap W_{\sigma(i)}| =\max\limits_{j\in\{0,1\}}|V_i\cap W_j|$.
This defines a map $\sigma:\{0,1\}\to\{0,1\}$.
Since $V_i\cap W_0$ and $V_i\cap W_1$ partition $V_i$, \eqref{eq:high-part-sizes} gives
\[
|V_i\cap W_{\sigma(i)}|
\ge \frac{|V_i|}{2}
\ge \frac12\left(
\left\lfloor\frac n2\right\rfloor-k+1
\right).
\]

Suppose that $V_i\cap W_{1-\sigma(i)}\ne\emptyset$.
Since $G_0$ contains all edges between
$W_0$ and $W_1$, we obtain
\[
\begin{aligned}
e(G_0[V_i])\ge
|V_i\cap W_{\sigma(i)}||V_i\cap W_{1-\sigma(i)}|
&\ge |V_i\cap W_{\sigma(i)}|\\
&\ge \frac12\left(\left\lfloor\frac n2\right\rfloor-k+1\right)\\
&>f(k-1,k-1)
\end{aligned}
\]
for sufficiently large $n$.
This contradicts \eqref{eq:high-partition-internal}.
Hence $V_i\subseteq W_{\sigma(i)}$ for every
$i\in\{0,1\}$.

If $\sigma(0)=\sigma(1)=j$, then $V(G_0)=V_0\cup V_1\subseteq W_j$, contradicting $W_{1-j}\ne\emptyset$.
Therefore $\sigma(0)\ne\sigma(1)$, so $\sigma$ is a permutation of $\{0,1\}$.
Since both partitions cover $V(G_0)$, $V_i\subseteq W_{\sigma(i)}$ imply $V_i=W_{\sigma(i)}$ for every $i\in\{0,1\}$.
\end{proof}

By Claim~\ref{partitions}, after relabeling $W_0$ and $W_1$, we may assume that
$V_i=W_i$ for $i\in\{0,1\}$. By symmetry, we may further assume that $J\subseteq G_0[V_1]$. 
Note that $\nu(J)=k-2$, so we can choose a matching $Q\subseteq E(J)$ of size $k-2$.
Choose edges $e_i\in M_i$ for $i\in\{0,1\}$ such that $V(e_i)\cap V(Q)=\emptyset$.

Since $\alpha_0\neq\alpha_1$ and neither occurs on $G_0$, the graph $G_0+e_0+e_1\in \calR(c,K_n)$.
We apply Lemma~\ref{lem:211} to show that $G_0+e_0+e_1$ admits  a copy of $H_{s,t}$.
Since $G_0[V_0,V_1]\cong T_2(n)$, conditions~(i) and~(ii) hold with $B_0=B_1=\emptyset$ and $U_i=V_i$ for $i\in\{0,1\}$.

Write $e_0=xy$. Since $B_0=B_1=\emptyset$, we have $x\in V_0\setminus B_0$.
Moreover, $x$ is adjacent in $G_0+e_0+e_1$ to every vertex of $V_1$, and $e_1$ is disjoint from $V(Q)$.
Thus $M=Q\cup\{e_1\}$ is a matching of size $k-1$ in $(G_0+e_0+e_1)\bigl[(N_{G_0+e_0+e_1}(x)\cap V_1)\setminus B_1\bigr]$.

Since $G_0[V_0]$ has no edges, $d_{(G_0+e_0+e_1)[V_0]}(x)=1$, and therefore $d_{(G_0+e_0+e_1)[V_0]}(x)+|M|=1+(k-1)=k\ge s$.
(\ref{long}) in condition~(iii) follows as well.
Lemma~\ref{lem:211} now gives a copy of $H_{s,t}$ centered at $x$ in $G_0+e_0+e_1$, a contradiction.
\end{proof}

\section{Proof of Theorem~\ref{thm:main}}
In this section we prove Theorem~\ref{thm:main}. We begin by deriving the lower bound for $\ar(n,H_{s,t})$, then establish its upper bound. For ease of reading, we restate the theorem below:\vspace{2mm}

\noindent\textbf{Theorem~1.1.} 
{Let $s\geq 0$ and $t\geq 1$ be fixed integers with $s+t\geq2$. There exists $n_0(H_{s,t})\in \mathbb{N}$ such that for all $n\geq n_0(H_{s,t})$,
$\ar(n, H_{s,t})= t_2(n) + (s+t-2)^2 + 1$.}

\begin{proof}[\bf Proof of Theorem~\ref{thm:main}]
Put $k=s+t$. We first prove the lower bound.

Let $G$ be obtained from $T_2(n)$ by adding all edges between two disjoint sets $A$ and $B$ in one partite set, where $|A|=|B|=k-2$.
Assign a distinct color to each edge of $G$, and color all remaining edges of $K_n-E(G)$ with one additional color.
This coloring uses $t_2(n)+(k-2)^2+1$ colors.

Suppose for a contradiction that this coloring contains a rainbow copy of $H_{s,t}$ with center $z$.
Since all edges outside $G$ have the same color, this copy contains at most one edge outside $G$.
Its $k$ cycles are edge-disjoint, so at least $k-1$ of them lie entirely in $G$.

Both $G-A$ and $G-B$ are bipartite, so every odd cycle in $G$ meets both $A$ and $B$.
Since $A\cap B=\emptyset$, we may assume by symmetry that $z\notin A$.
These $k-1$ cycles are pairwise vertex-disjoint outside $z$, and hence meet $A$ in at least $k-1$ distinct vertices.
This contradicts $|A|=k-2$. Therefore, the coloring contains no rainbow copy of
$H_{s,t}$, which proves $\ar(n,H_{s,t})\ge t_2(n)+(k-2)^2+1.$

We now prove the upper bound.

Suppose to the contrary that there is an edge-coloring $c$ of $K_n$ that contains no rainbow copy of
$H_{s,t}$ with at least $t_2(n)+(k-2)^2+2$ colors and $n\geq n_0(H_{s,t})$, where $n_0(H_{s,t})\gg n_1(H_{s,t})$. 

\medskip
\noindent\textbf{Case 1: $k=2$.}

Choose a representing graph $G\in\calR(c,K_n)$. Then $G$ is $H_{s,t}$-free.
By Lemma~\ref{thm:HQL}, we obtain
\[
t_2(n)+2
\le |c(E(K_n))|
=e(G)
\le \operatorname{ex}(n,H_{s,t})
=t_2(n)+1,
\]
a contradiction. Thus $\ar(n,H_{s,t})\le t_2(n)+1$.

\medskip
\noindent\textbf{Case 2: $k\geq 3$.}

If every representing graph in $\mathcal R(c,K_n)$ has minimum degree at least $\lfloor \frac n2\rfloor$,
then by Lemma~\ref{lem:high-degree}, the coloring $c$ uses at most $t_2(n)+(k-2)^2+1$ colors, contradicting the assumption that $c$ uses at least $t_2(n)+(k-2)^2+2$ colors. We may therefore suppose that there exists a representing graph $L^n\in \mathcal R(c,K_n)$ with $\delta(L^n)<\lfloor \frac n2\rfloor$. Thus there is a vertex $u_n\in V(K_n)$ satisfying $d_{L^n}(u_n)\le \lfloor \frac n2\rfloor-1$.

Let $G^n=K_n$ and $ G^{n-1}=G^n-u_n$, and let $c_{n-1}$ be the coloring of $G^{n-1}$  inherited from $c$.
Since $L^n$ contains exactly one edge of every color, deleting $u_n$ can remove at most $d_{L^n}(u_n)$ distinct colors from the coloring. Consequently, $G^{n-1}$ has at least $t_2(n)+(k-2)^2+2 -\lfloor\frac n2\rfloor+1$ colors. Since $t_2(n)-t_2(n-1) = \left\lfloor\frac{n}{2}\right\rfloor$, the edge-colored complete graph $G^{n-1}$ admits at least $t_2(n-1)+(k-2)^2+3$ distinct colors.

If every representing graph in $\mathcal R(c_{n-1},G^{n-1})$ has minimum degree at least $\lfloor \frac {n-1}{2} \rfloor$,
then Lemma~\ref{lem:high-degree} implies that $c_{n-1}$ uses at most $t_2(n-1)+(k-2)^2+1$ colors, which yields a contradiction.  Hence there exists a representing graph $L^{n-1}\in \mathcal R(c_{n-1},G^{n-1})$ and a vertex $u_{n-1}\in V(G^{n-1})$ such that $d_{L^{n-1}}(u_{n-1})\le\left\lfloor\frac{n-1}{2}\right\rfloor-1$.

Repeating this argument, we may construct a sequence of edge-colored complete graphs
\[
G^n,G^{n-1},\ldots,G^{n-\ell}
\]
such that the number of colors of $G^{n-\ell}$ is at least $t_2(n-\ell)+(k-2)^2+2+\ell$, which is based on $n_0(H_{s,t})\gg n_1(H_{s,t})$. Since an edge-coloring of $G^{n-\ell}$ has at most $\binom{n-\ell}{2}$ colors, this will yield a contradiction for large $\ell$.

Therefore $\ar(n,H_{s,t}) \leq t_2(n)+(s+t-2)^2+1$.
This completes the proof.

\end{proof}

\section*{Disclosure statement}
The author did not report any potential conflict of interest.
\section*{Data availability}
No data is available during the current study.

\end{document}